\documentclass[12pt]{article}

\usepackage[small]{titlesec}

\usepackage{float}

\usepackage{graphicx}

\usepackage{amsmath,amssymb}

\usepackage{array}

\newcommand{\mathscr}[1]{\ensuremath{\mathbf{#1}}}

\usepackage{amsthm}
\numberwithin{equation}{section}
\numberwithin{figure}{section}
\theoremstyle{plain}
\newtheorem {satz}{Theorem}[section]

\newtheorem {prop}[satz]{Proposition}
\newtheorem {kor}[satz]{Corollary}
\newtheorem {conj}[satz]{Conjecture}
\newtheorem {problem}{Problem}
\theoremstyle{definition}
\newtheorem {deff}[satz]{Definition}
\theoremstyle{remark}
\newtheorem {remark}[satz]{Remark}

\begin{document}
\date{}

\title{\Large {\bf Combinatorial properties of the $K3$ surface: 
Simplicial blowups and slicings}}

\author{Jonathan Spreer \& Wolfgang K\"{u}hnel}

\maketitle
\subsubsection*{\centering Abstract}
{\small
The $4$-dimensional abstract Kummer variety $K^4$ with 16 nodes leads to the $K3$ surface by resolving the 16 singularities. Here we present a simplicial realization of this minimal resolution. Starting with a minimal $16$-vertex triangulation of $K^4$ we resolve its $16$ isolated singularities -- step by step -- by simplicial blowups. As a result we obtain a $17$-vertex triangulation of the standard PL $K3$ surface. A key step is the construction of a triangulated version of the mapping cylinder of the Hopf map from the real projective 3-space onto the 2-sphere with the minimum number of vertices. Moreover we study simplicial Morse functions and the changes of their levels between the critical points. In this way we obtain slicings through the $K3$ surface of various topological types.} \\

{\small \textbf{MSC 2000: } Primary 57Q15; 
Secondary 14J28, 
14E15, 
57Q25, 
52B70 

\textbf{Key words:} intersection form, $K3$ surface, Kummer variety, combinatorial manifold, combinatorial pseudo manifold, resolution of singularities, simplicial Hopf map.}

\section{Introduction}
\label{klassTheorie}
Triangulations of manifolds with few vertices have been a growing subject of research during the last years. This is due to new computer facilities which allow calculations and even computer experiments with a list of simplices on, say, up to 50 vertices or more. Here we are dealing with {\it combinatorial $d$-manifolds} which are $d$-dimensional simplicial complexes such that the link of every $i$-simplex is a triangulated $(d-i-1)$-dimensional standard PL-sphere. For a {\it combinatorial $d$-pseudo manifold with isolated singularities} we require that the link of each vertex is a combinatorial $(d-1)$-manifold, not necessarily a sphere. Not all triangulated pseudo manifolds satisfy this property. It turned out that there is a triangulated 5-sphere with only 20 vertices which is not combinatorial, see \cite{BjornerLutz}. This example is not even a combinatorial pseudo manifold.

\medskip
The Problem of finding a combinatorial version of an abstract $d$-pseudo manifold is not trivial. Especially, if some additional properties such as vertex minimality is required. It is well known that there are the following operations in the class of combinatorial manifolds in order to solve this problem: {\it Products} and {\it connected sums}. The products require a simplicial subdivision of prisms but that is available. In algebraic geometry there is a third operation on a certain type of pseudo manifolds, namely, the {\it resolution of singularities}. A fourth operation would be a combinatorial version of {\it Dehn twists}. If these could be applied to simply connected combinatorial 4-manifolds we could make progress towards a solution of some interesting problems:

\begin{problem}
	\label{prob2}
	Find a pair of orientable PL $d$-manifolds $(M_1,M_2)$ such that
	\begin{description}
		\item[(i)] $M_1$ and $M_2$ are not homeomorphic,
		\item[(ii)] there are combinatorial triangulations of $M_1$ and $M_2$ with $n$ vertices but not with $n-1$ vertices,
		\item[(iii)] the $f$-vector of such an $n$-vertex triangulation is unique for both $M_1$ and $M_2$.
	\end{description}
	The entries of the $f$-vector are defined as the numbers $f_i$ of $i$-dimensional simplices of the triangulation.
\end{problem}

\begin{problem}
	\label{prob1}
	Find two concrete combinatorial triangulations of a $4$-manifold such that the underlying PL manifolds are homeomorphic but not PL homeomorphic. It is known that some compact topological $4$-manifolds admit exotic PL structures. Furthermore any combinatorial triangulation induces a unique PL structure and thus a unique smooth structure.
\end{problem}

\medskip
Concerning Problem \ref{prob2} there are pairs of non-orientable and orientable surfaces with the same minimum number of vertices, e.g., the two surfaces with $\chi = -10$ admit triangulations with the $f$-vector $(12,66,44)$ but no smaller triangulations. Moreover, the existence of pairs of non-homeomorphic lens spaces with the same minimum number of vertices is known due to Brehm and Swiatkowski \cite{BrehmSwiatkowski}. However, so far no such pair of concrete combinatorial manifolds was constructed. Concerning Problem \ref{prob1} it is well known that in a topological classification of simply connected 4-manifolds the relevant pieces are $\mathbb{C}P^2$ (with two orientations), $S^2 \times S^2$ and the $K3$ surface (with two orientations). However, for topological $4$-manifolds it can happen that there are possibly many distinct PL-structures. There is a method to construct exotic PL-structures on $4$-manifolds using {\it Akbulut corks}: Akbulut and Yasui investigated bounded submanifolds of a $4$-manifold $M$. These so-called corks can be cut out and glued back into the original manifold, thus changing the PL-type of $M$ (see \cite{Akbulut91FakeCompContr4Mfld} and \cite{Akbulut08CorksPlugsExoticStruc}). However, applying Akbulut corks to combinatorial manifolds requires more experiments. For $\mathbb{C}P^2$ and $S^2 \times S^2$ we have standard triangulations. For the $K3$ surface we have one optimal triangulation with the minimum number of 16 vertices \cite{Casella01TrigK3MinNumVert} but so far the PL type has not finally been identified. Presumably it is the standard structure of the classical $K3$ surface.

In this article we describe a purely combinatorial version of resolving ordinary nodes or double points in real dimension 4. In particular we describe this procedure for the $K3$ surface as a resolution of the Kummer variety with 16 nodes. ``Purely combinatorial'' here means that we are dealing with simplicial complexes (or subdivisions of such) with a relatively small number of vertices such that topological properties or modifications can be recognized or carried out by an efficient computer algorithm. The construction itself is fairly general. We are going to illustrate it for the example of the $K3$ surface as a desingularization of what we call a {\it Kummer variety}, following \cite{Spanier1956}. In particular we describe a straightforward and ``canonical'' procedure how a concrete triangulation of the $K3$ surface with a small number of vertices and with the classical PL structure can be obtained. As we will see in Chapter \ref{other} this procedure also gives some insights to Problem \ref{prob2}. In principle such a procedure seems to be possible in any even dimension.

\bigskip
For all this, computer algorithms are employed and implemented in the GAP-system \cite{GAP4}. Here, a key operation is the concept of {\it bistellar flips} due to Pachner \cite{Pachner1987} that establishes PL-homeomorphism on a combinatorial level. A GAP program due to Bj\"orner and Lutz \cite{BjornerLutz} implements a heuristical algorithm which reduces the number of vertices of a given combinatorial manifold without changing its PL-type. Since this process is not deterministic its character is rather experimental and needs a lot of computer calculation. Nonetheless we can use this concept together with some theoretical lower bounds to get closer to a solution of

\begin{problem}
	\label{prob3}
	For any given abstract compact PL $d$-manifold find the minimum number $n$ of vertices for a combinatorial triangulation of it, and find out which topological invariants are related to this number.
	
	For pseudo manifolds admitting some combinatorial triangulation we have the same problem.
\end{problem}

\section{The Kummer variety and the $K3$ surface}
\label{sec:k4andk3}

An abstract $d$-dimensional Kummer variety $ K^d = \mathbb{T}^d \big/_{x \sim -x}$ can be interpreted as the $d$-dimensional torus modulo involution \cite{Spanier1956}. It is a $d$-dimensional {\it flat orbifold} in the sense that the neighborhood of any point of $K^d$ is a quotient of Euclidean $d$-space by an orthogonal group. Topologically $K^d$ can be seen as a pseudo manifold with $2^d$ isolated singularities which are the fixed points of the involution. A typical neighborhood of a singularity is a cone over a real projective $(d-1)$-space where the apex represents the singularity. Thus, any combinatorial triangulation of $ K^d $ needs at least $ 2^d $ vertices as a kind of {\it absolute vertices} \cite{F'ary1977}. In more concrete terms a series of minimal triangulations of $K^d$ for any $d\geq3$ has been given in \cite{Kuhnel1986}. These combinatorial pseudo manifolds are $2$-neighborly (i.e., any two vertices are joined by an edge) and highly symmetric with a transitive automorphism group of order $ (d+1)! \cdot 2^d $. Moreover they contain a specific combinatorial real projective space $\mathbb{R}P^{d-1}$ with $2^d-1$ vertices as each vertex link. This vertex link happens to coincide with a 2-fold non-branched quotient of the vertex link of a series of combinatorial $d$-tori with $2^{d+1}-1$ vertices \cite{Kuhnel1988} which presumably has the minimum possible number of vertices among all combinatorial $d$-tori.

\medskip
In particular we have a minimal $2$-neighborly $16$-vertex triangulation of the $4$-dimensional Kummer variety which will be denoted by $ (K^4)_{16} $. A few of its properties are the following: The $f$-vector is given by $f = (16,120,400,480,192)$, the Euler characteristic is $ \chi (K^4) = 8 $, and the integral homology groups are \begin{equation} H_{*}(K^4) = (\mathbb{Z},0,\mathbb{Z}^6 \oplus (\mathbb{Z}_2)^5,0,\mathbb{Z}). \end{equation} Its intersection form is even of rank $6$ and signature $0$. We use an integer vertex labeling ranging from $1$ to $16$. The automorphism group of order $5! \cdot 2^4 = 1920$ is generated by two permutations as follows
\begin{eqnarray}
	& \langle &(1,7,12)(2,8,11)(3,10,16)(4,9,15),\nonumber \\ &&(1,9,10,14,16,8,7,3)(2,13,12,6,15,4,5,11) \quad \rangle .\nonumber
\end{eqnarray}
The complex coincides with the orbit $(1,2,4,8,16)_{192}$ (cf.\ \cite{Kuhnel1986} where the labeling is chosen as $0,1,2,\ldots,15$ instead of $1,2,3,\ldots,16$).

\medskip
The $K3$ surface, on the other hand, is a prime (i.e.,\ indecomposable by non-trivial connected sums, see \cite{Donaldson1986}) compact oriented connected and simply connected $4$-manifold, admitting a unique smooth or PL-structure. By Freedman's theorem \cite{Freedman1982} it is, up to homeomorphism, uniquely determined by its intersection form. The Euler characteristic is $ \chi (K3) = 24 $, the integral homology groups are
\begin{equation}
	H_{*}(K3) = (\mathbb{Z},0,\mathbb{Z}^{22},0,\mathbb{Z})
\end{equation}
and the intersection form is even of rank $22$ and signature $16$. In a suitable basis it is represented by the unimodular matrix
\begin{equation}
	\mathbb{E}_8 \oplus \mathbb{E}_8 \oplus 3 \begin{pmatrix} 0&1\\1&0\\ \end{pmatrix}
\end{equation}
where $\mathbb{E}_8$ is given by
{\small \begin{equation}
	\mathbb{E}_8 = 
	\begin{pmatrix}
		2&-1&0&0&0&0&0&0 \\
		-1&2&-1&0&0&0&0&0 \\
		0&-1&2&-1&0&0&0&0 \\
		0&0&-1&2&-1&0&0&0 \\
		0&0&0&-1&2&-1&0&-1 \\
		0&0&0&0&-1&2&-1&0 \\
		0&0&0&0&0&-1&2&0 \\
		0&0&0&0&-1&0&0&2 \\
	\end{pmatrix}.
\end{equation}}
This makes the $K3$ surface distinguished from the topological point of view. Also from the combinatorial point of view this 4-manifold is fairly special since the data $n=16, \chi = 24$ coincides with the case of equality in the generalized Heawood inequality
\begin{equation}\label{eq:Heawood}
	{{n-4}\choose 3} \geq 10(\chi(M)-2)
\end{equation}
which holds for any combinatorial $n$-vertex triangulation of any compact 4-manifold $M$, see \cite{Kuehnel94ManSkelConvPolyt}, \cite[4B]{Kuehnel95TightPolySubm}. Inequality (\ref{eq:Heawood}) is also a partial solution to Problem \ref{prob3} in the introduction. Equality can occur only for 3-neighborly triangulations, i. e.\ for which any triple of vertices determines a triangle of the triangulation. Consequently, the $f$-vector has to start with $\big(n,{n\choose 2},{n \choose 3}\big)$ in this case. In other words: Any combinatorial triangulation of the $K3$ surface has at least 16 vertices (the same number as required for the Kummer variety $K^4$), and one with precisely 16 vertices must necessarily be 3-neighborly (or {\it super-neighborly}). Such a $3$-neighborly, vertex minimal $16$-vertex triangulation of a PL manifold homeomorphic with the $K3$ surface $(K3)_{16}$ was found by Casella and the second author in \cite{Casella01TrigK3MinNumVert}. The $ f $-vector is $f = (16,120, 560, 720, 288)$, observe the 3-neighborliness $f_2 = 560 = {16 \choose 3}$. Its automorphism group is isomorphic to the affine linear group $\operatorname{AGL}(1,16)$ and is generated by two permutations as follows:
\begin{eqnarray}
	 & \langle & (1,3,8,4,9,16,15,2,14,12,6,7,13,5,10),\nonumber \\ 
	 & & (1,11,16)(2,10,14)(3,12,13)(4,9,15)(5,7,8) \quad \rangle \nonumber
\end{eqnarray}
This group of order $16 \cdot 15 = 240$ acts 2-transitively on the set of vertices $(1, \ldots , 16)$ of $(K3)_{16}$. The triangulation $(K3)_{16}$ itself is defined as the union of the orbits $( 1, 2, 3, 8, 12 )_{240}$ and $( 1, 2, 5, 8, 14 )_{48}$ under this permutation group, see \cite{Casella01TrigK3MinNumVert} where the labeling is chosen as $0,1,2,\ldots,15$ instead of $1,2,3,\ldots,16$.

\section{The Hopf $\sigma$-process}
\label{sec:blowups}

By the {\it Hopf $\sigma$-process} we mean the blowing up process of a point and, simultaneously, the resolution of nodes or ordinary double points of a complex algebraic variety. This was described by H.~Hopf in \cite{Hopf1951}, compare \cite{Hirzebruch1953} and \cite{Hauser2000}. From the topological point of view the process consists of cutting out some subspace and gluing in some other subspace. In complex algebraic geometry one point is replaced by the projective line $\mathbb{C}P^1 \cong S^2$ of all complex lines through that point. This is often called {\it blowing up} of the point. In general the process can be applied to non-singular 4-manifolds and yields a transformation of a manifold $M$ to $M \# (+\mathbb{C}P^2)$ or $M \# (-\mathbb{C}P^2)$, depending on the choice of an orientation. The same construction is possible for nodes or ordinary double points (a special type of singularities), and also the ambiguity of the orientation is the same for the blowup process of a node. Similarly it has been used in arbitrary even dimension by Spanier \cite{Spanier1956} as a so-called {\it dilatation process}. In the particular case of the 4-dimensional Kummer variety with 16 nodes a result of Hironaka \cite{Hironake1964} states that the singularities of a $4$-dimensional Kummer variety $K^4$ can be resolved into a smooth manifold, birationally equivalent to $K^4$. It is also well known that the minimal resolution of the $4$-dimensional Kummer variety is a $K3$ surface. This raises the question whether it is possible to carry out the Hopf $\sigma$-process in the purely combinatorial category. In this case one would have to cut out a certain neighborhood $A$ of each of the singularities and to glue in an appropriate simplicial complex $B$. 

\medskip
The spaces $A_i$ which have to be cut out are the following: The Kummer variety $K^4$ is the quotient of a $4$-dimensional torus $\mathbb{T}^4 = \mathbb{R}^4/_{\mathbb{Z}^4}$ by the central involution $\sigma : x \mapsto -x$ with precisely $16$ fixed points $ x_i $, $1 \leq i \leq 16 $. Let $X_i$ be a suitable neighborhood of $x_i$, then $\sigma$ acts on $ X = \mathbb{T}^4 \backslash \bigcup X_i $ without fixed points. The involution $\sigma$ acts as the antipodal map on each connected component of $\partial X$. Therefore the quotient of $\partial X_i$ is a projective space $\mathbb{R}P^3$ of dimension $3$ for each $ 1 \leq i \leq 16$, and the quotient of $X_i$ itself is a cone over it which we denote by $A_i$. Thus the quotient $ \widetilde{X} = X/_\sigma$ is a manifold having $16$ disjoint copies of $\mathbb{R}P^3$ as its boundary, and the quotient $K^4 = \mathbb{T}^4/_\sigma$ contains the disjoint subsets $A_1, \ldots, A_{16}$ as neighborhoods of the 16 singularities.

\medskip
The spaces $B_i$ which have to be glued in are the following: The Hopf map $h: S^3 \to \mathbb{C}P^1 $ induces a map $ \tilde{h}: \mathbb{R}P^3 \to \mathbb{C}P^1 $ since the Hopf map identifies antipodal pairs of points. We consider the cylinder $C = \mathbb{R}P^3 \times \left[ 0,1 \right] $ with the identification along the bottom of the cylinder by an equivalence relation $\sim$ defined by $(x,0) \sim (\tilde{h}(x),0)$. The quotient $\tilde{C} = C/_\sim $ is a manifold with boundary $\mathbb{R}P^3$. If we identify the boundary of $ \widetilde{X} $ with the union of the boundaries of $16$ copies $B_1, \ldots, B_{16}$ of $\widetilde{C}$ we get a closed manifold $S$. Alternatively each $B_i$ can be seen as a copy of $(\mathbb{C}P^2 \setminus B^4)/_\sigma$ where the involution $\tilde\sigma : \mathbb{C}P^2 \rightarrow \mathbb{C}P^2$ is defined by $\tilde\sigma[z_0,z_1,z_2] = [-z_0,z_1,z_2]$ with a fixed point set consisting of the point $[1,0,0]$ at the centre of the Ball $B^4$ and the polar projective line $z_0 = 0$. Spanier \cite{Spanier1956} proved that $S$ is in fact a $K3$ surface. Our main result is a simplicial realization of this construction. In principle one can expect that such a combinatorial construction is possible but there are a number of technical difficulties to overcome. One of the problems is to make the procedure efficient and to keep the number of vertices sufficiently small at each intermediate step.

\section{From $K^4$ to $K3$: Combinatorial resolution of the 16 singularities}
\label{sec:K4toK3}

Our goal is to construct a simplicial version of the $K3$ surface out of $(K^4)_{16}$ by a combinatorial version of Spanier's dilatation process. More precisely we find a way to cut out a certain simplicial version of $A_i$ and to glue in a simplicial version of $B_i$. We prefer a description of $B_i$ as the mapping cylinder of the Hopf map $\tilde{h}$, defined on $\mathbb{R}P^3 = \partial B^4/_{\tilde\sigma}$. In the combinatorial setting this is possible if the corresponding boundaries are combinatorially isomorphic, i. e. if they are equal up to a relabeling of the vertices. However, in general the boundaries are PL-homeomorphic but not combinatorially isomorphic. This is the main difficulty here. Therefore we need an efficient procedure to change the combinatorial type while preserving the PL-homeomorphism type of the manifold. One possibility of such a procedure is the well established concept of bistellar moves. Therefore we start with a short review on bistellar moves.

\begin{deff}{(Pachner's bistellar moves, see \cite{Pachner1987})\\}
	Let $M$ be a $d$-dimensional simplicial complex, and let $A$ be a $(d-i)$-face of $M$, where $ 0 \leq i \leq d $. If $\operatorname{lk}_{M}(A)$ is the boundary complex $\partial B$ of an $i$-simplex $B$ that is not a face of $M$, the operation $\Phi_{A}$ on $M$ defined by
	\begin{equation}
		\Phi_{A} ( M ) := (M\backslash (A \ast \partial B)) \cup (\partial A\ast B) 
	\end{equation}
	is called a \textit{bistellar $i$-move} or \textit{bistellar $i$-flip}. Similarly we have the {\it reverse bistellar $i$-flip} $\Phi_A^{-1}$ which can also be interpreted as a $(d-i)$-flip.
\end{deff}

\begin{figure}[h]
	\begin{center}
		\includegraphics[width=\textwidth]{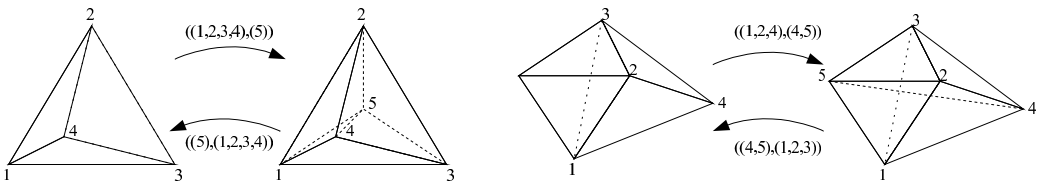}
	\end{center}
	\caption{$3$-dimensional bistellar moves}
	\label{bm}
\end{figure}

For $d=3$ all flips and reverse flips are shown in Fig. \ref{bm}. Two simplicial complexes $K$ and $L$ are called \textit{combinatorially isomorphic} or just \textit{isomorphic}, if they are equal up to a relabeling of the vertices. $K$ and $L$ are called \textit{bistellarly equivalent}, if there exists a sequence of bistellar flips from $K$ to a complex $K'$ such that $K'$ is isomorphic to $L$. This concept of bistellar flips has been a useful tool in several ways:

\begin{enumerate}
	\item By a theorem of Pachner \cite{Pachner1987} two combinatorial manifolds are PL homeomorphic if and only if the triangulations are bistellarly equivalent. \item From a practical point of view bistellar moves allow a reduction of the number of vertices of a given triangulation without changing its PL homeomorphism type. Many examples have been investigated and many small triangulations of $3$- and $4$-manifolds have been found using this technique, see \cite{Lutz1999} or \cite{Lutz2008b}.
	\item It is possible to decide whether two given complexes are PL homeomorphic by finding a connecting sequence of bistellar flips. This has been successful in many cases even if it cannot be excluded that the algorithm does not terminate. \item In particular it is possible to decide whether a given simplicial complex is a combinatorial manifold: One just has to examine the PL homeomorphism types of all links. \item These algorithms are implemented in a GAP-Program, see \cite{Lutz2008a}. 
\end{enumerate}

\subsection*{A triangulated mapping cylinder of the Hopf map $\tilde{h} : \mathbb{R}P^3 \rightarrow \mathbb{C}P^1$ with the minimum number of vertices}
\label{boundedCP2}

From the topology of the complex projective plane it is fairly clear that one can construct a triangulation of $\mathbb{C}P^2$ from a triangulated version of the Hopf map $h : S^3 \rightarrow S^2$. Conversely, every triangulation of $\mathbb{C}P^2$ contains implicitly a triangulation of the Hopf map (possibly with collapsing of certain simplices) by considering a neighborhood of a triangulated $\mathbb{C}P^1$ inside the triangulation.

\begin{satz}[Madahar and Sarkaria \cite{Madahar2000}]
	There is a simplicial version of the Hopf map $h : S^3 \rightarrow S^2$ with the minimum number of $12$ vertices for $S^ 3$ which are mapped in triplets onto the $4$-vertex $S^2$. From this simplicial Hopf map one can reconstruct the unique $9$-vertex triangulation of $\mathbb{C}P^2$ which was known before, see {\rm \cite{Kuehnel95TightPolySubm}.}
\end{satz}

Roughly the procedure for the construction of a triangulated $\mathbb{C}P^2$ is the following: 

\begin{enumerate}
	\item Find a simplicial subdivision of the mapping cylinder of the Hopf map which is a triangulated $\mathbb{C}P^2$ minus an open 4-ball. 
	\item Close it up on top by a suitable simplicial 4-ball. 
	\item Finally reduce the number of vertices by bistellar flips as far as possible. 
\end{enumerate}

For our purpose here we can follow an analogous procedure:

\begin{enumerate}
	\item Find a simplicial version of the Hopf map 	$ \tilde{h} : \mathbb{R}P^3 \rightarrow S^2$.
	\item Find a simplicial subdivision 	of the mapping cylinder $\widetilde{C}$ which is nothing but a triangulated complex 	projective plane with one hole modulo the involution $\tilde\sigma$. 	There is one boundary component 	which is homeomorphic to $\mathbb{R}P^3$.
	\item Finally reduce the number of vertices by bistellar flips 	as far as possible. 	It is well known that any combinatorial triangulation of 	$\mathbb{R}P^3$ has at least 11 vertices \cite{Walkup1970}. 	Therefore 11 is the minimum also for the space we are looking for.
\end{enumerate}

\begin{satz} 
	There is an $11$-vertex triangulation 	of the mapping cylinder of the Hopf map 	$\tilde{h} : \mathbb{R}P^3 \rightarrow S^2$ such 	that all vertices and edges are contained in the boundary. 	This is the minimum possible number of vertices 	since it is the minimum already for the boundary. 
\end{satz}

\begin{proof} 
On the boundary of $\mathbb{C}P^2 \setminus B^4$ the involution $\sigma$ coincides with $\tilde\sigma$ and leads to a twofold quotient map $S^3 \rightarrow \mathbb{R}P^3$. From this it is clear that a triangulated version of the Hopf map from $\mathbb{R}P^3$ onto $S^2$ requires a simplicial Hopf map $h : S^3 \rightarrow S^2$ which is centrally symmetric on $S^3$, i.e., which is invariant under $\sigma$. Therefore we need to construct a centrally symmetric triangulation of $S^3$ first. This should allow a simplicial fibration by Hopf fibres.

\medskip
For the construction we start with two regular hexagons (2-polytopes) $P_1,P_2$ in the plane and take the product polytope \cite[p.10]{Ziegler} $P := P_1 \times P_2.$ The vertices will be denoted by $a_{ij}$ where $i,j$ are ranging from 1 to 6. The facets of $P$ are $6 + 6$ hexagonal prisms where one of them has vertices $a_{11},\ldots,a_{16}$ on top and $a_{21},\ldots,a_{26}$ on bottom. The subcomplex $\partial P_1 \times \partial P_2 \subset \partial P$ is the standard $(6 \times 6)$-grid torus as a subcomplex decomposing $\partial P$ into two solid tori, one on each side of the torus. One of the squares has vertices $a_{11},a_{12},a_{21},a_{22}$, see Figure \ref{fig:gridtorus} where the labeling is simply $ij$ instead of $a_{ij}$. For a simplicial version we need to subdivide the prisms. In a first step we subdivide each square in the torus by the main diagonal, as indicated in Figure \ref{fig:gridtorus}. Next we introduce one extra vertex $b_i$ at the centre of the six prisms on one side and $c_i$ at the centre of the six prisms on the other side, $i = 1,\ldots,6$. That is to say, $b_1,\ldots, b_6$ represent the core of one solid torus and $c_1,\ldots, c_6$ the core of the other. Furthermore we introduce the pyramids from each $b_i$ and $c_i$ to the 12 triangles of each corresponding prism. Finally the remaining holes are closed by copies of the join of the edge between two adjacent centre vertices and the edge of a hexagon. Typical tetrahedra of this type are $\langle b_1 b_2 a_{11} a_{12}\rangle$ and $\langle c_1 c_2 a_{11} a_{21}\rangle$. This procedure is carried out for each of the two solid tori, see Figure \ref{fig:Hopfsphere}.

\medskip
Thus, we get a centrally symmetric triangulation $S^3_{cs}$ of the 3-sphere with $48$ vertices, with $2 \cdot (6 \cdot (12+6))=216$ tetrahedra and with an automorphism group $G$ of order 144.

\medskip
On this triangulation of $S^3$ we define the simplicial Hopf map $h_{cs}: S^3_{cs} \to S^2$ by the following identifications:

{\small \begin{eqnarray*}
	\{a_{ij}\ | \ j - i \equiv 1 \ (6)\} & \mapsto & a_1 \nonumber \\
	\{a_{ij}\ | \ j - i \equiv 2 \ (6)\} & \mapsto & a_2 \nonumber \\
	\{a_{ij}\ | \ j - i \equiv 3 \ (6)\} & \mapsto & a_3 \nonumber \\
	\{a_{ij}\ | \ j - i \equiv 4 \ (6)\} & \mapsto & a_4 \nonumber \\
	\{a_{ij}\ | \ j - i \equiv 5 \ (6)\} & \mapsto & a_5 \nonumber \\
	\{a_{ij}\ | \ j - i = 0 \} & \mapsto & a_6 \nonumber \\
	\{b_{i} \} & \mapsto & b \nonumber \\
	\{c_{i} \} & \mapsto & c \nonumber \\
\end{eqnarray*}}

\begin{figure}[p]
	\begin{center}
	 \includegraphics[width=0.8\textwidth]{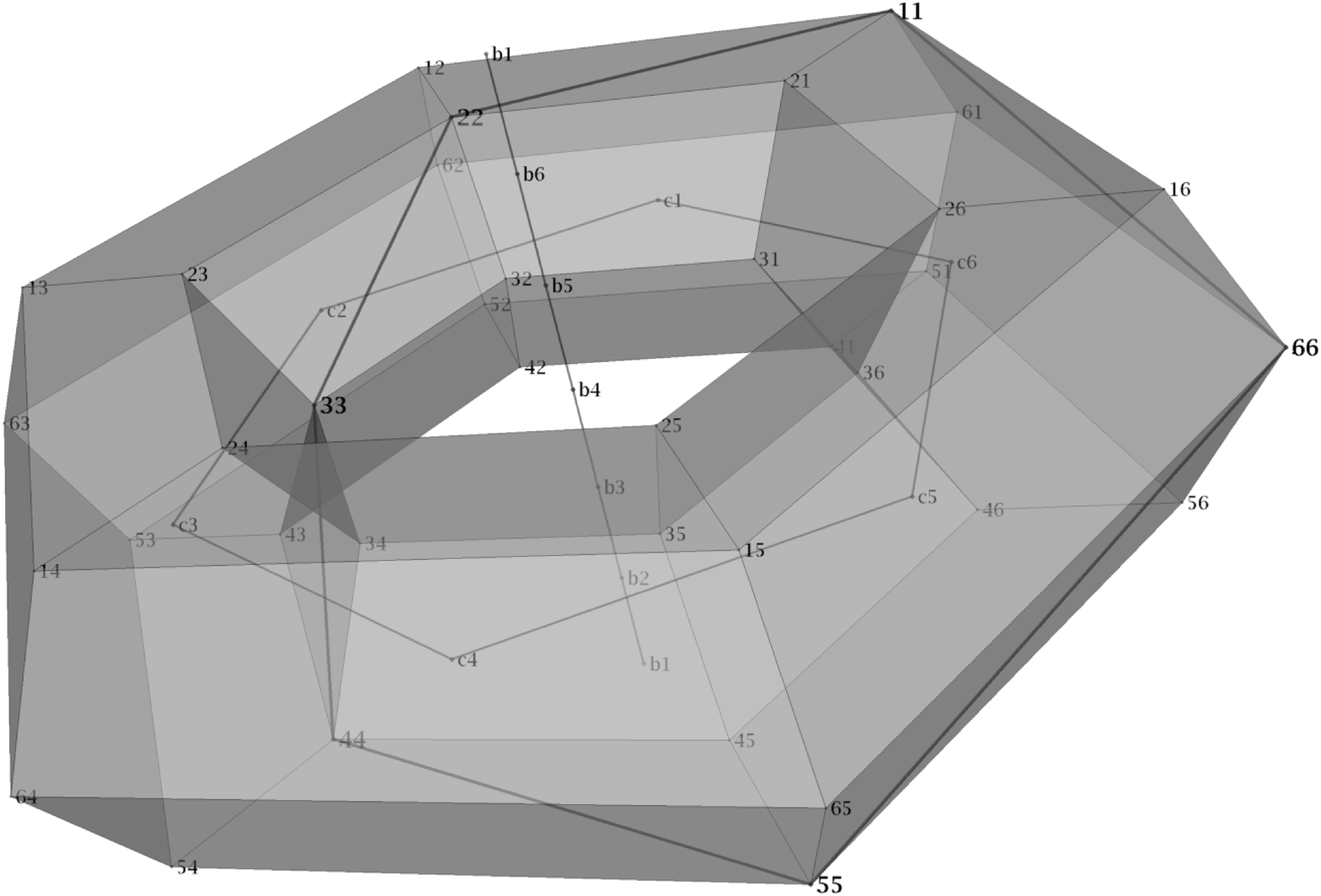}
	 \caption{A solid torus as half of $S^3_{cs}$ with two Hopf fibres. \label{fig:Hopfsphere}}
	 
	 \bigskip
	 \bigskip
	 \bigskip
	 \includegraphics[height=7cm]{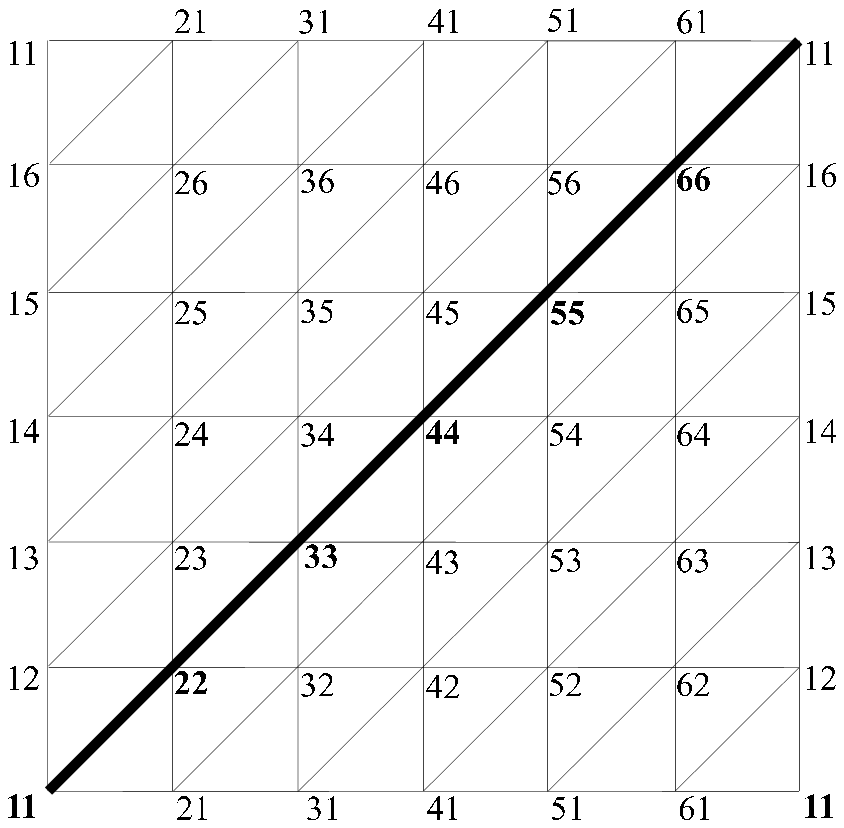}
	 \caption{Combinatorial $(6 \times 6)$-grid torus with Hopf fibres. \label{fig:gridtorus}}
	\end{center}
\end{figure}

The image is a simplicial $2$-sphere with $8$ vertices, namely, a double pyramid from $b$ and $c$ over the hexagon $a_1,a_2,a_3,a_4,a_5,a_6$. Note that the antipodal map $$\sigma: \ a_{ij} \mapsto a_{i+3,j+3}, \ \ b_{i} \mapsto b_{i+3}, \ \ c_{i} \mapsto c_{i+3} $$ (all indices taken modulo 6) is compatible with the simplicial Hopf map. By construction the quotient $\mathbf{P} = S^3_{cs} /_\sigma$ is a $24$-vertex triangulation of $\mathbb{R}P^3$. The automorphism group is a normal subgroup of $G$ of index $2$. It follows that this triangulated $\mathbb{R}P^3$ again allows a simplicial version of the Hopf map $\tilde{h}: \mathbf{P} \to S^2$.

The image of the torus in Figure \ref{fig:Hopfsphere} under $h$ forms the hexagon $(a_1,a_2,a_3,a_4,a_5,a_6)$, and each of the solid tori on each side gets mapped to a cone over it. A suitable simplicial decomposition $C$ of the cylinder $ \mathbf{P} \times \left[ 0, 1 \right] $ is compatible with the projection map $$\tilde{h} \times \{0\} : \mathbf{P} \times \{0\} \to S^2 \times \{0\}$$ $$\qquad \quad (x,0) \mapsto (\tilde{h}(x),0) $$ on the bottom of $ C $ and leads to a triangulated mapping cylinder $C/_\sim$. Its boundary is PL homeomorphic to the link of any vertex of $(K^4)_{16}$.

\medskip
For the purpose of a better handling of the blowup process we computed a reduced version of $ \mathscr{C} \cong_{PL} C /_\sim $ by bistellar flips. In this reduced version the boundary is isomorphic to a vertex minimal triangulation of $\mathbb{R}P^3$ with the $f$-vector $ (11,51,80,40) $. Moreover, the boundary $\partial \mathscr{C}$ is bistellarly equivalent to $\operatorname{lk}_{(K^4)_{16}} ( v )$ for any vertex $v$ which will be needed for the construction of a triangulated $K3$ surface. On 11 vertices $1,2,\ldots,11$ this complex is the following: 
\begin{eqnarray}
	\mathscr{C} & = & \big\langle( 1, 3, 5, 6, 11 ),( 2, 3, 5, 6, 11 ),
	( 2, 4, 5, 6, 11 ),( 2, 3, 6, 9, 11 ), ( 3, 6, 7, 9, 11 ),
	\nonumber \\
	&& ( 1, 3, 6, 7, 11 ), ( 6, 7, 8, 10, 11 ), ( 1, 6, 7, 10, 11 ),
	( 4, 6, 8, 9, 11 ), ( 6, 7, 8, 9, 11 ),
	\nonumber \\
	&& ( 2, 4, 6, 9, 11 ), ( 1, 2, 3, 5, 8 ),( 1, 2, 3, 5, 11 ), 
	( 1, 5, 7, 8, 9 ), ( 1, 2, 5, 7, 8 ),
	\nonumber \\
	&& ( 1, 4, 5, 9, 11 ),( 4, 5, 7, 9, 11 ), ( 1, 4, 5, 7, 9 ), 
	( 1, 2, 4, 5, 11 ), ( 1, 2, 4, 5, 7 ),
	\nonumber \\
	&& ( 3, 4, 7, 8, 11 ), ( 1, 3, 4, 7, 8 ), ( 1, 2, 3, 8, 11 ), 
	( 1, 3, 7, 8, 11 ),( 1, 2, 8, 10, 11 ),
	\nonumber \\
	&& ( 1, 7, 8, 10, 11 ), ( 1, 2, 7, 8, 10 ),( 4, 7, 8, 9, 11 ), 
	( 1, 4, 7, 8, 9 ), ( 1, 2, 4, 9, 11 ) \big\rangle
	\nonumber 
\end{eqnarray}
It has the $f$-vector $ (11,51,107,95,30) $, and its boundary $\partial \mathscr{C}$ contains the complete $1$-skeleton of $\mathscr{C}$. In particular, $\mathscr{C}$ is vertex minimal since the boundary $\mathbb{R}P^3$ requires already at least 11 vertices for any simplicial triangulation \cite{Walkup1970}.
\end{proof}

\begin{remark} 
	By starting with the product polytope of two $3k$-gons containing a $3k \times 3k$-grid torus one can similarly obtain a simplicial version of the Hopf map from the lens space $L(k,1)$ to $S^2$. Furthermore the same procedure as above can be carried out for the corresponding mapping cylinder.
\end{remark}

\subsection*{Simplicial blowups}
\label {ssec:scBlowups}

A PL version of the Hopf $\sigma$-process from Section~\ref{sec:k4andk3} is the following: We cut out the star of one of the singular vertices which is nothing but a cone over a triangulated $\mathbb{R}P^3$. This corresponds to the space $A_i$ above. The boundary of the resulting space is this triangulated $\mathbb{R}P^3$ and is therefore PL homeomorphic with the boundary of the triangulated mapping cylinder $\mathscr{C}$ from Section~\ref{sec:blowups} which corresponds to the space $B_i$. Then we cut out $A_i$ and glue in $B_i$ by an appropriate PL homeomorphism, as indicated in Section~\ref{sec:k4andk3}.

\medskip
For a combinatorial version with concrete triangulations, however, we face the problem that these two triangulations are not isomorphic. This implies that before cutting out and gluing in we have to modify the triangulations by bistellar flips until they coincide. This computation is provided by the GAP-program BISTELLAR which is available from \cite{Lutz2008a}. For more information about GAP, see \cite{GAP4}.

\begin{deff}{(Resolution of singularities in PL topology)\\}
	Let $v$ be a singular vertex of a PL $4$-pseudo manifold $M$ with a compact neighborhood $A$ of the type ``cone over an $\mathbb{R}P^3$'' and let $ \phi : \partial A \to \partial\mathscr{C}$ be a PL-homeomorphism. A \emph{PL resolution of the singularity $v$} is given by the following construction
	\begin{equation}
		M \mapsto \widetilde{M} := (M \setminus A^\circ) \cup_{\phi} \mathscr{C}.
	\end{equation}
	We will refer to this operation as a {\it PL blowup} of $v$.
\end{deff}

\begin{deff}{(Simplicial resolution of singularities)\\}
	Let $v$ be a vertex of a $4$-pseudo manifold $M$ whose link is isomorphic with the particular $11$-vertex triangulation of $\mathbb{R}P^3$ which is given by the boundary complex of the triangulated $\mathscr{C}$ above. Let $\psi : \operatorname{lk}(v) \rightarrow \partial\mathscr{C}$ denote such an isomorphism. A \emph{simplicial resolution of the singularity $v$} is given by the following construction
	\begin{equation}
		M \mapsto \widetilde{M} := (M \setminus \operatorname{star}(v)^\circ) \cup_{\psi} \mathscr{C}.
	\end{equation}
	We will refer to this operation as a {\it simplicial blowup} or just a {\it blowup} of $v$.
\end{deff}

\medskip
Since in either case both parts are glued together along their PL-homeomorphic boundaries, the resulting complex is closed, and the construction of $\widetilde{M}$ is well defined. $\widetilde{M}$ is a closed pseudo manifold and the number of singular points in $\widetilde{M}$ is the number of singular points in $M$ minus one. In particular we can apply this to $M = (K^4)_{16}$ and then repeat the procedure for the resulting spaces until the last singularity disappears. We can now prove the following main result:

\begin{satz}
	There is a $17$-vertex triangulation of the $K3$ surface $(K3)_{17}$ with the standard PL structure which can be constructed from $(K^4)_{16}$ by a sequence of bistellar flips and, in between, by $16$ simplicial blowups.
\end{satz}

The proof of the theorem is constructive and will be given in the form of an algorithm. From the construction it is clear that the resulting PL manifold is PL homeomorphic with the classical $K3$ surface, not only homeomorphic.

\medskip
Let $\widetilde{K}_i$, $ 0 \leq i \leq 16 $ be the $4$-dimensional Kummer variety after the $i$-th blowup. Since we have to modify its combinatorial type repeatedly our notation will not distinguish here between two different complexes after bistellar flips. Furthermore, let $\mathscr{C}$ be the bounded complex from Section~\ref{boundedCP2} and $Q_i$ the intersection form of $ \widetilde{K}_i $.

\medskip
We start with a singular vertex $v \in \widetilde{K}_{i} $. In general, its link is not isomorphic to $\partial \mathscr{C}$. Thus, we have to modify $\widetilde{K}_{i} \backslash \operatorname{star} (v)$ in a suitable way to yield a complex which allows a simplicial blowup with the space $\mathscr{C}$. This is accomplished by modifying $\partial(\widetilde{K}_{i} \backslash \operatorname{star} (v)) = \operatorname{lk}(v)$ with respect to the combinatorial structure of the complex. Even though in general we cannot claim that this must be possible in any case, in this particular case we were able to find fairly short sequences of bistellar moves realizing this modification at any of the 16 steps. The sequences were found using the approach from \cite{Lutz2008a}.

\medskip
Once $\partial(\widetilde{K}_{i} \backslash \operatorname{star} (v))$ is isomorphic to $\partial \mathscr{C} $ we can perform the simplicial blowup and gain $\widetilde{K}_{i+1}$ as indicated above. Note, that in each step we can perform the blowup in two significantly different ways coresponding to the choice of orientation of $\mathscr{C}$. For the verification of the choice of the right orientation we compute the intersection form $Q_{i+1}$ and check that $|\operatorname{sign}(Q_{i+1})| = | \operatorname{sign}(Q_{i}) + 1 | $ holds (note, that $Q(K^4)=0$ and $Q(K3)=\pm 16$).

\medskip
Due to the various modifications above, the resulting complex $ \widetilde{K}_{i+1} $ will be considerably bigger than $ \widetilde{K}_{i} $. Thus, we use bistellar flips to reduce it before repeating the same operation for the remaining singularities. In every step the signature of the intersection form and the second Betti number will increase by one and the number of singularities will decrease by one. Also, the torsion part of $H_2 ( K^4 )$ will gradually decline. This, however, depends on the order of the blowing up process of the singularities. It follows that the resulting complex is a triangulation of the $K3$ surface with the right intersection form and with the standard PL structure. The algorithm is written in GAP \cite{GAP4}. For the computation of the intersection form we use polymake \cite{Joswig00Polymake}.

The smallest complex (with respect to the $f$-vector) we were able to obtain by bistellar moves is a $17$-vertex version of the $K3$-surface which will be denoted by $(K3)_{17}$. Its facets as well as some basic properties are listed in Table \ref{tab:facetList}.

\begin{table}
	\begin{center}
		\tiny
		\begin{tabular}{lllllll}
			$(K3)_{17} = $&$\langle 1\,2\,3\,8\,13 \rangle$, &
			$\langle 1\,2\,3\,8\,14 \rangle$, &
			$\langle 1\,2\,3\,12\,13 \rangle$, &
			$\langle 1\,2\,3\,12\,15 \rangle$, &
			$\langle 1\,2\,3\,14\,15 \rangle$, &
			$\langle 1\,2\,4\,7\,13 \rangle$,\\
			
			&$\langle 1\,2\,4\,7\,15 \rangle$, &
			$\langle 1\,2\,4\,13\,15 \rangle$, &
			$\langle 1\,2\,5\,6\,9 \rangle$, &
			$\langle 1\,2\,5\,6\,14 \rangle$, &
			$\langle 1\,2\,5\,9\,17 \rangle$, &
			$\langle 1\,2\,5\,14\,15 \rangle$,\\
			
			&$\langle 1\,2\,5\,15\,17 \rangle$, &
			$\langle 1\,2\,6\,8\,14 \rangle$, &
			$\langle 1\,2\,6\,8\,15 \rangle$, &
			$\langle 1\,2\,6\,9\,16 \rangle$, &
			$\langle 1\,2\,6\,15\,17 \rangle$, &
			$\langle 1\,2\,6\,16\,17 \rangle$,\\
			
			&$\langle 1\,2\,7\,8\,11 \rangle$, &
			$\langle 1\,2\,7\,8\,15 \rangle$, &
			$\langle 1\,2\,7\,11\,13 \rangle$, &
			$\langle 1\,2\,8\,10\,11 \rangle$, &
			$\langle 1\,2\,8\,10\,13 \rangle$, &
			$\langle 1\,2\,9\,16\,17 \rangle$,\\
			
			&$\langle 1\,2\,10\,11\,13 \rangle$, &
			$\langle 1\,2\,12\,13\,15 \rangle$, &
			$\langle 1\,3\,4\,5\,8 \rangle$, &
			$\langle 1\,3\,4\,5\,17 \rangle$, &
			$\langle 1\,3\,4\,6\,10 \rangle$, &
			$\langle 1\,3\,4\,6\,12 \rangle$,\\
			
			&$\langle 1\,3\,4\,8\,9 \rangle$, &
			$\langle 1\,3\,4\,9\,10 \rangle$, &
			$\langle 1\,3\,4\,12\,17 \rangle$, &
			$\langle 1\,3\,5\,7\,11 \rangle$, &
			$\langle 1\,3\,5\,7\,14 \rangle$, &
			$\langle 1\,3\,5\,8\,16 \rangle$,\\
			
			&$\langle 1\,3\,5\,11\,16 \rangle$, &
			$\langle 1\,3\,5\,14\,15 \rangle$, &
			$\langle 1\,3\,5\,15\,17 \rangle$, &
			$\langle 1\,3\,6\,10\,12 \rangle$, &
			$\langle 1\,3\,7\,8\,11 \rangle$, &
			$\langle 1\,3\,7\,8\,14 \rangle$,\\
			
			&$\langle 1\,3\,8\,9\,12 \rangle$, &
			$\langle 1\,3\,8\,11\,16 \rangle$, &
			$\langle 1\,3\,8\,12\,13 \rangle$, &
			$\langle 1\,3\,9\,10\,12 \rangle$, &
			$\langle 1\,3\,12\,15\,17 \rangle$, &
			$\langle 1\,4\,5\,8\,16 \rangle$,\\
			
			&$\langle 1\,4\,5\,11\,16 \rangle$, &
			$\langle 1\,4\,5\,11\,17 \rangle$, &
			$\langle 1\,4\,6\,7\,12 \rangle$, &
			$\langle 1\,4\,6\,7\,15 \rangle$, &
			$\langle 1\,4\,6\,10\,15 \rangle$, &
			$\langle 1\,4\,7\,12\,13 \rangle$,\\
			
			&$\langle 1\,4\,8\,9\,16 \rangle$, &
			$\langle 1\,4\,9\,10\,14 \rangle$, &
			$\langle 1\,4\,9\,14\,16 \rangle$, &
			$\langle 1\,4\,10\,14\,16 \rangle$, &
			$\langle 1\,4\,10\,15\,16 \rangle$, &
			$\langle 1\,4\,11\,16\,17 \rangle$,\\
			
			&$\langle 1\,4\,12\,13\,17 \rangle$, &
			$\langle 1\,4\,13\,15\,16 \rangle$, &
			$\langle 1\,4\,13\,16\,17 \rangle$, &
			$\langle 1\,5\,6\,9\,13 \rangle$, &
			$\langle 1\,5\,6\,13\,14 \rangle$, &
			$\langle 1\,5\,7\,10\,12 \rangle$,\\
			
			&$\langle 1\,5\,7\,10\,14 \rangle$, &
			$\langle 1\,5\,7\,11\,12 \rangle$, &
			$\langle 1\,5\,9\,11\,13 \rangle$, &
			$\langle 1\,5\,9\,11\,17 \rangle$, &
			$\langle 1\,5\,10\,12\,14 \rangle$, &
			$\langle 1\,5\,11\,12\,13 \rangle$,\\
			
			&$\langle 1\,5\,12\,13\,14 \rangle$, &
			$\langle 1\,6\,7\,8\,14 \rangle$, &
			$\langle 1\,6\,7\,8\,15 \rangle$, &
			$\langle 1\,6\,7\,10\,12 \rangle$, &
			$\langle 1\,6\,7\,10\,16 \rangle$, &
			$\langle 1\,6\,7\,14\,16 \rangle$,\\
			
			&$\langle 1\,6\,9\,11\,13 \rangle$, &
			$\langle 1\,6\,9\,11\,14 \rangle$, &
			$\langle 1\,6\,9\,14\,16 \rangle$, &
			$\langle 1\,6\,10\,15\,16 \rangle$, &
			$\langle 1\,6\,11\,13\,14 \rangle$, &
			$\langle 1\,6\,15\,16\,17 \rangle$,\\
			
			&$\langle 1\,7\,10\,14\,16 \rangle$, &
			$\langle 1\,7\,11\,12\,13 \rangle$, &
			$\langle 1\,8\,9\,10\,11 \rangle$, &
			$\langle 1\,8\,9\,10\,12 \rangle$, &
			$\langle 1\,8\,9\,11\,16 \rangle$, &
			$\langle 1\,8\,10\,12\,13 \rangle$,\\
			
			&$\langle 1\,9\,10\,11\,14 \rangle$, &
			$\langle 1\,9\,11\,16\,17 \rangle$, &
			$\langle 1\,10\,11\,13\,14 \rangle$, &
			$\langle 1\,10\,12\,13\,14 \rangle$, &
			$\langle 1\,12\,13\,15\,16 \rangle$, &
			$\langle 1\,12\,13\,16\,17 \rangle$,\\
			
			&$\langle 1\,12\,15\,16\,17 \rangle$, &
			$\langle 2\,3\,4\,6\,12 \rangle$, &
			$\langle 2\,3\,4\,6\,16 \rangle$, &
			$\langle 2\,3\,4\,7\,14 \rangle$, &
			$\langle 2\,3\,4\,7\,16 \rangle$, &
			$\langle 2\,3\,4\,12\,14 \rangle$,\\
			
			&$\langle 2\,3\,5\,6\,12 \rangle$, &
			$\langle 2\,3\,5\,6\,16 \rangle$, &
			$\langle 2\,3\,5\,12\,16 \rangle$, &
			$\langle 2\,3\,7\,14\,16 \rangle$, &
			$\langle 2\,3\,8\,13\,14 \rangle$, &
			$\langle 2\,3\,9\,10\,13 \rangle$,\\
			
			&$\langle 2\,3\,9\,10\,14 \rangle$, &
			$\langle 2\,3\,9\,11\,14 \rangle$, &
			$\langle 2\,3\,9\,11\,16 \rangle$, &
			$\langle 2\,3\,9\,12\,13 \rangle$, &
			$\langle 2\,3\,9\,12\,16 \rangle$, &
			$\langle 2\,3\,10\,13\,14 \rangle$,\\
			
			&$\langle 2\,3\,11\,14\,16 \rangle$, &
			$\langle 2\,3\,12\,14\,15 \rangle$, &
			$\langle 2\,4\,5\,7\,10 \rangle$, &
			$\langle 2\,4\,5\,7\,11 \rangle$, &
			$\langle 2\,4\,5\,8\,10 \rangle$, &
			$\langle 2\,4\,5\,8\,12 \rangle$,\\
			
			&$\langle 2\,4\,5\,11\,12 \rangle$, &
			$\langle 2\,4\,6\,11\,12 \rangle$, &
			$\langle 2\,4\,6\,11\,16 \rangle$, &
			$\langle 2\,4\,7\,9\,14 \rangle$, &
			$\langle 2\,4\,7\,9\,15 \rangle$, &
			$\langle 2\,4\,7\,10\,13 \rangle$,\\
			
			&$\langle 2\,4\,7\,11\,16 \rangle$, &
			$\langle 2\,4\,8\,10\,12 \rangle$, &
			$\langle 2\,4\,9\,10\,13 \rangle$, &
			$\langle 2\,4\,9\,10\,14 \rangle$, &
			$\langle 2\,4\,9\,13\,15 \rangle$, &
			$\langle 2\,4\,10\,12\,14 \rangle$,\\
			
			&$\langle 2\,5\,6\,7\,10 \rangle$, &
			$\langle 2\,5\,6\,7\,12 \rangle$, &
			$\langle 2\,5\,6\,8\,10 \rangle$, &
			$\langle 2\,5\,6\,8\,14 \rangle$, &
			$\langle 2\,5\,6\,9\,16 \rangle$, &
			$\langle 2\,5\,7\,11\,12 \rangle$,\\
			
			&$\langle 2\,5\,8\,12\,14 \rangle$, &
			$\langle 2\,5\,9\,15\,16 \rangle$, &
			$\langle 2\,5\,9\,15\,17 \rangle$, &
			$\langle 2\,5\,12\,14\,15 \rangle$, &
			$\langle 2\,5\,12\,15\,16 \rangle$, &
			$\langle 2\,6\,7\,10\,13 \rangle$,\\
			
			&$\langle 2\,6\,7\,11\,12 \rangle$, &
			$\langle 2\,6\,7\,11\,13 \rangle$, &
			$\langle 2\,6\,8\,10\,15 \rangle$, &
			$\langle 2\,6\,10\,11\,13 \rangle$, &
			$\langle 2\,6\,10\,11\,17 \rangle$, &
			$\langle 2\,6\,10\,15\,17 \rangle$,\\
			
			&$\langle 2\,6\,11\,16\,17 \rangle$, &
			$\langle 2\,7\,8\,11\,15 \rangle$, &
			$\langle 2\,7\,9\,11\,14 \rangle$, &
			$\langle 2\,7\,9\,11\,15 \rangle$, &
			$\langle 2\,7\,11\,14\,16 \rangle$, &
			$\langle 2\,8\,10\,11\,15 \rangle$,\\
			
			&$\langle 2\,8\,10\,12\,14 \rangle$, &
			$\langle 2\,8\,10\,13\,14 \rangle$, &
			$\langle 2\,9\,11\,15\,17 \rangle$, &
			$\langle 2\,9\,11\,16\,17 \rangle$, &
			$\langle 2\,9\,12\,13\,16 \rangle$, &
			$\langle 2\,9\,13\,15\,16 \rangle$,\\
			
			&$\langle 2\,10\,11\,15\,17 \rangle$, &
			$\langle 2\,12\,13\,15\,16 \rangle$, &
			$\langle 3\,4\,5\,8\,17 \rangle$, &
			$\langle 3\,4\,6\,8\,9 \rangle$, &
			$\langle 3\,4\,6\,8\,11 \rangle$, &
			$\langle 3\,4\,6\,9\,15 \rangle$,\\
			
			&$\langle 3\,4\,6\,10\,15 \rangle$, &
			$\langle 3\,4\,6\,11\,13 \rangle$, &
			$\langle 3\,4\,6\,13\,16 \rangle$, &
			$\langle 3\,4\,7\,12\,14 \rangle$, &
			$\langle 3\,4\,7\,12\,17 \rangle$, &
			$\langle 3\,4\,7\,13\,16 \rangle$,\\
			
			&$\langle 3\,4\,7\,13\,17 \rangle$, &
			$\langle 3\,4\,8\,11\,17 \rangle$, &
			$\langle 3\,4\,9\,10\,15 \rangle$, &
			$\langle 3\,4\,11\,13\,17 \rangle$, &
			$\langle 3\,5\,6\,10\,12 \rangle$, &
			$\langle 3\,5\,6\,10\,15 \rangle$,\\
			
			&$\langle 3\,5\,6\,13\,15 \rangle$, &
			$\langle 3\,5\,6\,13\,16 \rangle$, &
			$\langle 3\,5\,7\,11\,15 \rangle$, &
			$\langle 3\,5\,7\,14\,15 \rangle$, &
			$\langle 3\,5\,8\,16\,17 \rangle$, &
			$\langle 3\,5\,9\,10\,15 \rangle$,\\
			
			&$\langle 3\,5\,9\,10\,17 \rangle$, &
			$\langle 3\,5\,9\,15\,17 \rangle$, &
			$\langle 3\,5\,10\,12\,16 \rangle$, &
			$\langle 3\,5\,10\,16\,17 \rangle$, &
			$\langle 3\,5\,11\,15\,16 \rangle$, &
			$\langle 3\,5\,13\,15\,16 \rangle$,\\
			
			&$\langle 3\,6\,7\,8\,9 \rangle$, &
			$\langle 3\,6\,7\,8\,15 \rangle$, &
			$\langle 3\,6\,7\,9\,15 \rangle$, &
			$\langle 3\,6\,8\,11\,15 \rangle$, &
			$\langle 3\,6\,11\,13\,15 \rangle$, &
			$\langle 3\,7\,8\,9\,13 \rangle$,\\
			
			&$\langle 3\,7\,8\,11\,15 \rangle$, &
			$\langle 3\,7\,8\,13\,14 \rangle$, &
			$\langle 3\,7\,9\,13\,17 \rangle$, &
			$\langle 3\,7\,9\,15\,17 \rangle$, &
			$\langle 3\,7\,12\,14\,15 \rangle$, &
			$\langle 3\,7\,12\,15\,17 \rangle$,\\
			
			&$\langle 3\,7\,13\,14\,16 \rangle$, &
			$\langle 3\,8\,9\,12\,13 \rangle$, &
			$\langle 3\,8\,10\,11\,16 \rangle$, &
			$\langle 3\,8\,10\,11\,17 \rangle$, &
			$\langle 3\,8\,10\,16\,17 \rangle$, &
			$\langle 3\,9\,10\,11\,14 \rangle$,\\
			
			&$\langle 3\,9\,10\,11\,16 \rangle$, &
			$\langle 3\,9\,10\,12\,16 \rangle$, &
			$\langle 3\,9\,10\,13\,17 \rangle$, &
			$\langle 3\,10\,11\,13\,14 \rangle$, &
			$\langle 3\,10\,11\,13\,17 \rangle$, &
			$\langle 3\,11\,13\,14\,16 \rangle$,\\
			
			&$\langle 3\,11\,13\,15\,16 \rangle$, &
			$\langle 4\,5\,7\,8\,13 \rangle$, &
			$\langle 4\,5\,7\,8\,16 \rangle$, &
			$\langle 4\,5\,7\,10\,13 \rangle$, &
			$\langle 4\,5\,7\,11\,16 \rangle$, &
			$\langle 4\,5\,8\,10\,15 \rangle$,\\
			
			&$\langle 4\,5\,8\,11\,12 \rangle$, &
			$\langle 4\,5\,8\,11\,17 \rangle$, &
			$\langle 4\,5\,8\,13\,15 \rangle$, &
			$\langle 4\,5\,9\,10\,13 \rangle$, &
			$\langle 4\,5\,9\,10\,15 \rangle$, &
			$\langle 4\,5\,9\,13\,15 \rangle$,\\
			
			&$\langle 4\,6\,7\,9\,12 \rangle$, &
			$\langle 4\,6\,7\,9\,15 \rangle$, &
			$\langle 4\,6\,8\,9\,14 \rangle$, &
			$\langle 4\,6\,8\,11\,14 \rangle$, &
			$\langle 4\,6\,9\,12\,14 \rangle$, &
			$\langle 4\,6\,11\,12\,14 \rangle$,\\
			
			&$\langle 4\,6\,11\,13\,17 \rangle$, &
			$\langle 4\,6\,11\,16\,17 \rangle$, &
			$\langle 4\,6\,13\,16\,17 \rangle$, &
			$\langle 4\,7\,8\,13\,16 \rangle$, &
			$\langle 4\,7\,9\,12\,14 \rangle$, &
			$\langle 4\,7\,12\,13\,17 \rangle$,\\
			
			&$\langle 4\,8\,9\,14\,16 \rangle$, &
			$\langle 4\,8\,10\,11\,12 \rangle$, &
			$\langle 4\,8\,10\,11\,15 \rangle$, &
			$\langle 4\,8\,11\,14\,15 \rangle$, &
			$\langle 4\,8\,13\,15\,16 \rangle$, &
			$\langle 4\,8\,14\,15\,16 \rangle$,\\
			
			&$\langle 4\,10\,11\,12\,15 \rangle$, &
			$\langle 4\,10\,12\,14\,16 \rangle$, &
			$\langle 4\,10\,12\,15\,16 \rangle$, &
			$\langle 4\,11\,12\,14\,15 \rangle$, &
			$\langle 4\,12\,14\,15\,16 \rangle$, &
			$\langle 5\,6\,7\,10\,12 \rangle$,\\
			
			&$\langle 5\,6\,8\,10\,15 \rangle$, &
			$\langle 5\,6\,8\,13\,14 \rangle$, &
			$\langle 5\,6\,8\,13\,15 \rangle$, &
			$\langle 5\,6\,9\,13\,16 \rangle$, &
			$\langle 5\,7\,8\,9\,13 \rangle$, &
			$\langle 5\,7\,8\,9\,17 \rangle$,\\
			
			&$\langle 5\,7\,8\,16\,17 \rangle$, &
			$\langle 5\,7\,9\,10\,13 \rangle$, &
			$\langle 5\,7\,9\,10\,17 \rangle$, &
			$\langle 5\,7\,10\,14\,16 \rangle$, &
			$\langle 5\,7\,10\,16\,17 \rangle$, &
			$\langle 5\,7\,11\,15\,16 \rangle$,\\
			
			&$\langle 5\,7\,14\,15\,16 \rangle$, &
			$\langle 5\,8\,9\,12\,13 \rangle$, &
			$\langle 5\,8\,9\,12\,17 \rangle$, &
			$\langle 5\,8\,11\,12\,17 \rangle$, &
			$\langle 5\,8\,12\,13\,14 \rangle$, &
			$\langle 5\,9\,11\,12\,13 \rangle$,\\
			
			&$\langle 5\,9\,11\,12\,17 \rangle$, &
			$\langle 5\,9\,13\,15\,16 \rangle$, &
			$\langle 5\,10\,12\,14\,16 \rangle$, &
			$\langle 5\,12\,14\,15\,16 \rangle$, &
			$\langle 6\,7\,8\,9\,16 \rangle$, &
			$\langle 6\,7\,8\,14\,16 \rangle$,\\
			
			&$\langle 6\,7\,9\,12\,16 \rangle$, &
			$\langle 6\,7\,10\,13\,17 \rangle$, &
			$\langle 6\,7\,10\,16\,17 \rangle$, &
			$\langle 6\,7\,11\,12\,13 \rangle$, &
			$\langle 6\,7\,12\,13\,17 \rangle$, &
			$\langle 6\,7\,12\,16\,17 \rangle$,\\
			
			&$\langle 6\,8\,9\,14\,16 \rangle$, &
			$\langle 6\,8\,11\,13\,14 \rangle$, &
			$\langle 6\,8\,11\,13\,15 \rangle$, &
			$\langle 6\,9\,11\,12\,13 \rangle$, &
			$\langle 6\,9\,11\,12\,14 \rangle$, &
			$\langle 6\,9\,12\,13\,16 \rangle$,\\
			
			&$\langle 6\,10\,11\,13\,17 \rangle$, &
			$\langle 6\,10\,15\,16\,17 \rangle$, &
			$\langle 6\,12\,13\,16\,17 \rangle$, &
			$\langle 7\,8\,9\,12\,16 \rangle$, &
			$\langle 7\,8\,9\,12\,17 \rangle$, &
			$\langle 7\,8\,12\,16\,17 \rangle$,\\
			
			&$\langle 7\,8\,13\,14\,16 \rangle$, &
			$\langle 7\,9\,10\,13\,17 \rangle$, &
			$\langle 7\,9\,11\,14\,15 \rangle$, &
			$\langle 7\,9\,12\,14\,15 \rangle$, &
			$\langle 7\,9\,12\,15\,17 \rangle$, &
			$\langle 7\,11\,14\,15\,16 \rangle$,\\
			
			&$\langle 8\,9\,10\,11\,16 \rangle$, &
			$\langle 8\,9\,10\,12\,16 \rangle$, &
			$\langle 8\,10\,11\,12\,17 \rangle$, &
			$\langle 8\,10\,12\,13\,14 \rangle$, &
			$\langle 8\,10\,12\,16\,17 \rangle$, &
			$\langle 8\,11\,13\,14\,15 \rangle$,\\
			
			&$\langle 8\,13\,14\,15\,16 \rangle$, &
			$\langle 9\,11\,12\,14\,15 \rangle$, &
			$\langle 9\,11\,12\,15\,17 \rangle$, &
			$\langle 10\,11\,12\,15\,17 \rangle$, &
			$\langle 10\,12\,15\,16\,17 \rangle$, &
			$\langle 11\,13\,14\,15\,16 \rangle$.
		\end{tabular}	
		\normalsize
		\caption{$17$-vertex triangulation of the $K3$ surface $(K3)_{17}$ with the standard PL structure. Its $f$-vector is $f ( K3 ) = (17,135, 610, 780, 312)$. Note, that the complex is not $2$-neighborly. \label{tab:facetList}}
	\end{center}
\end{table}

Further data as well as all $16$ steps of the dilatation process are available from the webpage of the first author, given at the end of this article. The source code itself is available upon request.

\medskip
So far we were not able to prove PL-equivalence to $(K3)_{16}$. However, since the given complex only has $17$ vertices this is most likely to be true, and further experiments will probably prove the following conjecture.

\begin{conj}
	\label{conj:main}
	$(K3)_{17}$ is PL homeo\-morphic to $(K3)_{16}$.
\end{conj}

\begin{remark}
	If Conjecture \ref{conj:main} is false this would imply that $(K3)_{16}$ is exotic. In this case, to our knowledge, $(K3)_{16}$ would be the first explicit triangulation with few vertices of a non-standard combinatorial $4$-manifold.
\end{remark}

\section{Critical Point Theory and Slicings}
\label{schnitteDurchK3}

The Morse theory for smooth functions defined on smooth manifolds is an important tool in topology. Similarly, the PL-structure of a $d$-dimensional combinatorial (pseudo-)manifold $M$ can be examined using an analogous concept of critical points of functions, defined on a triangulation of a manifold $M$, compare \cite{Kuehnel90TrigMnfFewVert}, \cite{Kuehnel95TightPolySubm}. In this section we describe a few computer experiments on triangulations of the $K3$ surface and the Kummer variety. As a result we obtain a picture of slices through these spaces by levels of perfect Morse functions in a PL version.

\begin{deff}
 \label{def:rsl}
 Let $M^d$ be a combinatorial manifold. A function $f:M \to \mathbb{R}$ is called \textit{regular simplexwise linear (rsl)} if $f(v) \neq f(w)$ for any two vertices $w \neq v$ and if $f$ is linear when restricted to an arbitrary simplex of the triangulation.

 A point $x \in M$ is said to be {\it critical} for an rsl-function $f:M \to \mathbb{R}$ if
 \[H_{\star} (M_x , M_x \backslash \{ x \} , F) \neq 0 \]
 where $M_x := \{ y \in M | f(y) \leq f(x) \}$ and $F$ is a field.

 	It follows that no point of $M$ can be critical except possibly for the vertices. If we fix an rsl-function $f$ and a vertex $v$ of $M$ we can define the {\it multiplicity vector} $\mathbf{m}(v,F)$ as the following $(d+1)$-tuple of integers:
	\begin{equation*}
		\mathbf{m}(v,F) := \big( \dim H_0 ( M_v , M_v \backslash \{ v \} , F) , \ldots , \dim H_d ( M_v , M_v \backslash \{ v \} , F) \big).
	\end{equation*}
	$v$ is called {\it critical of index $i$ and multiplicity $m$} if $\dim H_i ( M_v , M_v \backslash \{ v \} , F)= m > 0 $.
	\[ \sum \limits_{i=0}^{d} \dim H_i ( M_v , M_v \backslash \{ v \} , F)\]
	is called {\it total multiplicity } of $v$,
	\[ \mu_i (f,F) := \sum \limits_{v \in V} \dim H_i ( M_v , M_v \backslash \{ v \} , F) \]
	is referred to as the {\it number of critical points of index $i$} in $M$ (where $V$ denotes the set of vertices of $M$), and 	\[ \mu (f,F) := \sum \limits_{i = 0}^{d} \mu_i (f,F) \] is said to be {\it the number of critical points of $M$}. The multiplicity vector together with the number of critical points has to be considered for appropriately encoding the relevant information about the PL-structure of $M$ since, in contrast with the smooth case, higher multiplicities cannot be avoided in general.
\end{deff}

The classical Morse relation $\mu_i (f,F) \geq \operatorname{b}_i (M,F) $ is still true for any rsl-function $f$ on $M$ where $\operatorname{b}_i(M,F) = {\rm dim}H_i(M;F)$ denotes its $i$-th Betti number. Equality refers to the case of a tight or perfect function. If any rsl-function has this property we call the triangulation of $M$ a {\it tight triangulation} (cf. \cite{Kuehnel95TightPolySubm}). In particular $\mu_i$ does not depend on $f$ in the tight case. For some examples of multiplicities on $(K3)_{16}$ see Table \ref{tab:k3} below.

\begin{deff}
	\label{def:slicing}
	Let $M$ be an orientable combinatorial $d$-(pseudo-)manifold and let $f:M \to \mathbb{R}$ be an rsl-function. Then we call the pre-image $f^{-1} (x)$ a {\it slicing} of $M$ whenever $x \neq f(v)$ for any vertex $v \in M$.
\end{deff}

By construction a slicing is a PL $(d-1)$-manifold and we have $f^{-1} (x) \cong f^{-1} (y)$ whenever $f^{-1}[f(x),f(y)]$ contains no vertex, i.e., if no vertex is mapped into the interval $[f(x),f(y)]$. Note that any partition of the set of vertices $V = V_1 \dot{\cup} V_2 $ of $M$ already determines a slicing: Just define an rsl-function $g$ with $g(v) < g(w)$ for all $v \in V_1$ and $w \in V_2$ and look at a suitable pre-image. For an example of a slicing of a 3-manifold and a 3-pseudo manifold see \cite[Fig.9,Fig.11]{Kuehnel95TightPolySubm}.

\medskip
Since every combinatorial (pseudo-)manifold has a finite number of vertices there exist only a finite number of slicings. Hence, if $f$ is chosen carefully, the induced slicings admit a useful visualization of $M$.

\subsection*{Combinatorial Morse analysis on $(K3)_{16}$}
\label{ssec:morseAnaK3}

The 16-vertex triangulation of the $K3$ surface is a very special object in combinatorial topology:

\begin{enumerate}
	\item It satisfies equality in the {\it generalized Heawood inequality} (\ref{eq:Heawood}) for the number $n$ of vertices of 	a 4-manifold $M$ with Euler characteristic $\chi(M)$.
	\item It is {\it $3$-neighborly} or {\it super-neighborly} 	meaning that $f_2 = {n \choose 3}$. See \cite[Thm.5.8]{Swartz} for a characterization of all possible $g$-vectors of a triangulated $K3$ surface, starting with the minimum 
	$(g_0,g_1,g_2) = (1,10,55)$.
	\item It is the only known triangulation of a 4-manifold 	admitting an automorphism group acting {\it $2$-transitively} on the set of vertices	(besides the trivial case of the 6-vertex 4-sphere).
\end{enumerate}

\noindent
From the view-point of Morse theory this has the following consequence:

\begin{prop} 
	Any rsl-function $f$ defined on $(K3)_{16}$ is a perfect function in the sense that the total number of critical points is $24$. More precisely we have $\mu_0(f) = \mu_4(f) = 1$, $\mu_1(f) = \mu_3(f) = 0$ and $\mu_2(f) = 22$. This holds for any choice of a field $F$.
\end{prop}

\begin{proof}
	This follows from the fact that the 16-vertex triangulation of the $K3$ surface is a tight triangulation in the sense of \cite{Kuehnel95TightPolySubm}, \cite{KuhnelLutz}. The reason is that the triangulation is 3-neighborly which implies that there are no critical points of index 1: Any subset of vertices spans a connected and simply connected subset. The 2-neighborliness implies that every rsl-function has exactly one critical point of index 0. The rest follows from the duality $\mu_i(f) = \mu_{4-i}(-f)$ and the Poincar\'{e} relation $\mu_0 - \mu_1 + \mu_2 - \mu_3 + \mu_4 = \chi(K3) = 24$.
\end{proof}

\begin{kor} 
	Any rsl-function $f$ defined on $(K3)_{16}$ has a critical point of index $2$ with a multiplicity higher than $2$. More precisely $10$ possible critical vertices have to build up the second Betti number $22$. This holds for any choice of a field $F$.	
	
	Moreover, any slicing of an rsl-function on $(K3)_{16}$ is a connected $3$-manifold.
\end{kor}

\begin{proof} 
	The first part is obvious from the Morse inequality $\mu_2(f) \geq b_2(M) = 22$ and the fact that by the 3-neighborliness only the middle vertices (i.e., all but the three on top and the three on bottom) can be critical of index 2. For examples of multiplicity vectors see Table \ref{tab:k3} below. The second part follows from the fact that there is no critical point of index 1. If there were a disconnected level it would have to be modified into a connected level later, and this procedure requires a critical point of index 1 in between.
\end{proof}
 
\medskip
It may be interesting to see how the levels of such a function change when passing through a critical level. It does not seem to be known from differential topology what the possible levels can be for smooth perfect functions on the $K3$ surface. The standard embedding $(z_0,z_1,z_2,z_3) \mapsto (z_i\bar{z_j})_{ij}$ of a quartic surface in projective 3-space $K3 \rightarrow \mathbb{C}P^3 \rightarrow S^{14} \rightarrow \mathbb{R}^{15}$ induces smooth Morse functions by linear projections from 15-space to $\mathbb{R}$. However, in general these won't be perfect. It is well known that there is no tight smooth embedding or immersion of the $K3$ surface into any Euclidean space \cite{Thorbergsson}. Not too much seems to be known about possible slicings of perfect smooth Morse functions, defined on the $K3$ surface. In the PL case we have the following feature:

\medskip 
An rsl-function on $(K3)_{16}$ 
\begin{equation*} 
	f_{\Omega}: (K3)_{16} \to \left[ 0,1 \right] 
\end{equation*}
is essentially determined by a fixed ordering on the set of vertices $\Omega := \{ v_1 , \ldots , v_{16} \}$ determining the function $f_{\Omega}$ by the condition $0 = f_{\Omega}(v_1) < \ldots < f_{\Omega}(v_{16}) = 1$. Any slicing $f_{\Omega}^{-1}(\alpha)$ of a 4-manifold consists only of tetrahedra and $3$-dimensional prisms of type $\Delta^2 \times \left[ 0, 1 \right] $ (where $ \Delta^2 $ denotes a triangle), induced by proper sections with the $4$-simplices of $(K3)_{16}$. In many cases the topological type of $ f_{\Omega}^{-1}(\alpha) $ can be identified using standard techniques. Some of the slicings can be seen in advance:

\begin{itemize}
	\item \textbf{The $3$-torus:} Obviously there is a 3-torus as a slicing of the 4-torus. It can be arranged that this avoids all the 16 fixed points of the involution. Hence we have the same slicing in the Kummer variety and, by the purely local resolution procedure, also in the $K3$ surface.
	\item \textbf{The real projective $3$-space:} The link of any singular point in $K^4$ is a real projective $3$-space. By resolving the singularities we only change a neighborhood of these points. Thus, there are slicings in $(K3)_{16}$ separating such a neighborhood. These are homeomorphic with $\mathbb{R}P^3$.
	\item \textbf{The Poincar\'{e} homology sphere $\Sigma^3$:} There is a surgery description of the $K3$ surface showing the Poincar\'{e} homology sphere as a possible slicing (see \cite{Saveliev1999} for details). Even though this does not tell about the number of vertices which will be needed it turned out that a certain slicing of the 16-vertex triangulation is this manifold $\Sigma^3$, see below.
\end{itemize}

\begin{prop}
	As slicings of $(K3)_{16}$ we obtain at least the manifolds 	$$S^3,\mathbb{R}P^3,L(3,1),L(4,1),L(5,1),\Sigma^3$$ 	and a number of other space forms: The $3$-torus, the 	cube space, the octahedron space, 	the truncated cube space and the prism space $P(3)$. 	Here $\Sigma^3$ denotes the Poincar\'{e} homology sphere with 	a fundamental group of order $120$.
\end{prop}

\begin{proof}
	The permutation $$ ( 1,16)( 2,15)( 3,14)( 4,13)( 5,12)( 6,11)( 7,10)( 8, 9) $$ on the 16 vertices is an automorphism of $(K3)_{16}$, and we have $ f_{\{ 1, \ldots, 16 \} }^{-1} (\frac{1}{2}) \cong \mathbb{T}^3 $. Hence, we use this slicing as a starting point and analyze all possible slicings of $(K3)_{16}$ around this 3-torus in the middle.
	
	\medskip
	Since $(K3)_{16}$ is $3$-neighborly all slicings with $3$ or less vertices on one side are trivial (i. e. the slicing is a $3$-sphere). With four vertices on one side we have two possible situations. Either the tetrahedron formed by the four vertices is contained in the complex (in this case the slicing is clearly trivial) or it is not. In the latter case we have a slicing behind an empty tetrahedron. This type of slicing is a real projective $3$-space. Therefore a simplicial decomposition of the set $ |f_{\Omega}^{-1} (\left[ 0, \alpha \right] )| $, $ \frac{3}{15} < \alpha < \frac{4}{15} $ is PL-homeomorphic to the mapping cylinder $\mathscr{C}$ of the Hopf map $\tilde{h}$ in Section~\ref{sec:K4toK3}. Hence, we can find topological copies of $\mathscr{C}$ in $(K3)_{16}$ (which is not surprising).
	
	\medskip
	Neither the span of $ \{ 1 , \ldots , 8 \} $ nor the span of $ \{ 9 , \ldots , 16 \} $ contains a $4$-simplex of $ (K3)_{16} $. Thus, $5$ vertices on one side cannot induce a trivial slicing with a sphere but such with a lens space of type $\operatorname{L}(4,1)$, $\operatorname{L}(3,1)$ or $\operatorname{L}(2,1) = \mathbb{R}P^3$. In the case of $6$ or $7$ vertices we have the cube space, the octahedron space or the Poincar\'{e} homology sphere. $8$ vertices on each side result in the $3$-torus, the only non-spherical $3$-manifold in this series.
	
	\medskip
	For a complete list of the topological types of these slicings see Table \ref{tab:toptypesk3}.

	\begin{table}[H]
	{\small 
		\begin{center}
			\setlength{\extrarowheight}{0.1cm}
			\begin{tabular}{|l||cc||cl||cl|} 
				\hline
				&\multicolumn{2}{|l||}{$f_{\{ 1, \ldots ,5,7,6,8,9,11,10,12,\ldots ,16\}}$}&\multicolumn{2}{|l||}{$f_{\{ 2,\ldots ,7,1,8,9,16,10, \ldots ,15\}}$}&\multicolumn{2}{|l|}{$f_{\{ 1, \ldots ,5,7,6,8,9,11,10,12, \ldots ,16\}}$}\\ 
				\hline
				$f_i(v)$ & $v$ & $\mathbf{m}(v,\mathbb{F}_2)$ & $v$ & $\mathbf{m}(v,\mathbb{F}_2)$ & $v$ & $\mathbf{m}(v,\mathbb{F}_2)$\\
				\hline
				$0$&$1$&$(1,0,0,0,0)$&$2$&$(1,0,0,0,0)$&$1$&$(1,0,0,0,0)$ 	\\
				$\frac{1}{15}$&$2$&$(0,0,0,0,0)$&$3$&$(0,0,0,0,0)$&$2$&$(0,0,0,0,0)$ \\
				$\frac{2}{15}$&$3$&$(0,0,0,0,0)$&$4$&$(0,0,0,0,0)$&$3$&$(0,0,0,0,0)$ \\
				$\frac{3}{15}$&$4$&$(0,0,1,0,0)$&$5$&$(0,0,0,0,0)$&$4$&$(0,0,1,0,0)$ \\
				$\frac{4}{15}$&$5$&$(0,0,2,0,0)$&$6$&$(0,0,1,0,0)$&$5$&$(0,0,1,0,0)$ \\
				$\frac{5}{15}$&$7$&$(0,0,3,0,0)$&$7$&$(0,0,3,0,0)$&$7$&$(0,0,2,0,0)$ \\
				$\frac{6}{15}$&$6$&$(0,0,2,0,0)$&$1$&$(0,0,4,0,0)$&$6$&$(0,0,4,0,0)$ \\
				$\frac{7}{15}$&$8$&$(0,0,3,0,0)$&$8$&$(0,0,3,0,0)$&$8$&$(0,0,3,0,0)$ \\
				$\frac{8}{15}$&$9$&$(0,0,3,0,0)$&$9$&$(0,0,3,0,0)$&$9$&$(0,0,3,0,0)$ \\
				$\frac{9}{15}$&$11$&$(0,0,2,0,0)$&$16$&$(0,0,4,0,0)$&$11$&$(0,0,4,0,0)$\\
				$\frac{10}{15}$&$10$&$(0,0,3,0,0)$&$10$&$(0,0,3,0,0)$&$10$&$(0,0,2,0,0)$\\
				$\frac{11}{15}$&$12$&$(0,0,2,0,0)$&$11$&$(0,0,1,0,0)$&$12$&$(0,0,1,0,0)$\\
				$\frac{12}{15}$&$13$&$(0,0,1,0,0)$&$12$&$(0,0,0,0,0)$&$13$&$(0,0,1,0,0)$\\
				$\frac{13}{15}$&$14$&$(0,0,0,0,0)$&$13$&$(0,0,0,0,0)$&$14$&$(0,0,0,0,0)$\\
				$\frac{14}{15}$&$15$&$(0,0,0,0,0)$&$14$&$(0,0,0,0,0)$&$15$&$(0,0,0,0,0)$\\
				$1$&$16$&$(0,0,0,0,1)$&$15$&$(0,0,0,0,1)$&$16$&$(0,0,0,0,1)$\\ 
				\hline 
			\end{tabular}
			\\[0,2cm]
			\caption{Multiplicity vectors of the critical points of 
			$f_{\{ 1, \ldots ,5,7,6,8,9,11,10,12,\ldots ,16\}}$, $f_{\{ 2,\ldots ,7,1,8,9,16,10, \ldots ,15\}}$, $f_{\{ 1, \ldots ,5,7,6,8,9,11,10,12, \ldots ,16\}}: (K3)_{16} \to [0,1]$ \label{tab:k3}} 
		\end{center}}
	\end{table}
	
	\begin{table}[H]
		{\small \begin{center}
			\begin{tabular}{|r|c|l|}
				\hline
				$\alpha$&$f_\Omega^{-1} (\alpha)$&slicing in between\\
				\hline
				&&\\
				$\frac{1}{30}$&${S}^3$&$\{1\}$ and $\{2,\ldots ,16\}$\\
				&&\\
				$\frac{1}{10}$&${S}^3$&$\{1,2\}$ and $\{ 3,\ldots ,16\}$\\
				&&\\
				$\frac{1}{6}$&${S}^3$&$\{1,2,3\}$ and $\{4,\ldots ,16\}$\\
				&&\\
				$\frac{7}{30}$&${S}^3$&$\{2,3,4,5\}$ and $ \{1,6,\ldots ,16\}$\\
				&$\mathbb{R}P^3 $&$\{1,2,3,4\}$ and $\{5,\ldots ,16\}$\\
				&&\\
				$\frac{3}{10}$&$\operatorname{L}(4,1)$&$\{1,\ldots ,5\}$ and $\{ 6,\ldots ,16\}$\\
				&$\operatorname{L}(3,1)$&$\{2,\ldots ,6\}$ and $\{1,7,\ldots ,16\}$\\
				&$\mathbb{R}P^3$ &$\{1,2,3,5,6\}$ and $\{4,7,\ldots ,16\}$\\
				&&\\
				$\frac{11}{30}$&$\mathscr{C}^3$&$\{2,\ldots ,7\}$ and $\{1,8,\ldots ,16\}$\\
				&$\mathscr{O}^3$&$\{1,\ldots ,5,7\}$ and $\{ 6,8,\ldots ,16\} $\\
				&&\\
				$\frac{13}{30}$&$\Sigma^3$&$\{1,\ldots ,7\}$ and $\{8,\ldots ,16\}$\\
				&&\\
				$\frac{1}{2}$&$\mathbb{T}^3 $&$\{1,\ldots ,8\}$ and $\{9,\ldots
				,16\} $ \\[0.1cm] 
				\hline
			\end{tabular}
		\end{center}}
		\caption{Topological types of slicings of $(K3)_{16}$. Here $\Sigma^3$ denotes the Poincar\'{e} homology sphere, $\mathscr{C}^3$ the cube space and $ \mathscr{O}^3$ the octahedron space.\label{tab:toptypesk3}}
	\end{table}
	
	\bigskip
	Besides the symmetrical slicings of Table~\ref{tab:toptypesk3} we found a number of other 3-dimensional spherical space forms like the truncated cube space, the prism space $P(3)$ or the lens space $\operatorname{L}(5,1)$ as well as some orientable flat manifolds. Triangulations of such spaces were found in \cite{Lutz1999}. These can be used for comparison by bistellar flips.
\end{proof}

\subsection*{Combinatorial Morse analysis on $(K^4)_{16}$}
\label{ssec:morseAnaK4}

In this section we will use the field $ F := \mathbb{F}_2 $ because the Kummer variety has 2-torsion in the homology, see Equation (1.1). Since $(K^4)_{16}$ is not a combinatorial manifold we cannot apply critical point theory as easily as for the $K3$ surface. The reason is that now all vertex links are distinct from combinatorial $3$-spheres. This implies that duality does no longer hold. Moreover it has the following consequence somehow against our intuition about Morse theory: Slicings below a non-critical point do not necessarily have to be homeomorphic to the ones above the same non-critical point. Moreover, $(K^4)_{16}$ is not a tight triangulation. A tight triangulation of a simply connected space (manifold or not) must be 3-neighborly but $(K^4)_{16}$ is not because of $f_2 = 400 < {{16}\choose 3}$. In particular not all rsl-functions are perfect functions. 

\bigskip
From the $\mathbb{F}_2$-Betti numbers $b_0 = 1, \ b_1 = 0, \ b_2 = 11, \ b _3 = 5, \ b_4 = 1$ of the Kummer variety we expect that any rsl-functions has $18$ or more critical points, counted with multiplicity. The question is whether there is a perfect rsl-function on this triangulation which in addition fits the symmetry of the complex. This would be an excellent candidate for visualizing the space $(K^4)_{16}$ by various 3-dimensional slicings.

\begin{prop}
	As slicings associated with perfect functions on $(K^4)_{16}$ we obtain at least the manifolds
	$$\mathbb{R}P^3,\mathbb{R}P^3 \# \mathbb{R}P^3,
	\mathbb{R}P^3 \# \mathbb{R}P^3 \# \mathbb{R}P^3,
	S^2 \times S^1 \# \mathbb{R}P^3 \# \mathbb{R}P^3$$
	and the $3$-torus.
\end{prop}

\begin{proof}
	There is the following perfect rsl-function $f_{\{1 , \ldots , 16 \}}$ given by
	\[ \textstyle f_{\{1 , \ldots , 16 \}} : (K^4)_{16} \to \left[ 0 , 1
	\right]; 
	\quad i \mapsto \frac{i-1}{15}. \]
	As we already know, the first and the last slicing represent the 	link of the vertex $1$ (or $16$, resp.) and are, therefore, 	combinatorial real projective $3$-spaces. 	Furthermore, the middle slicing 	$f_{\{1 , \ldots , 16 \}}^{-1}(\frac{1}{2})$ 	is homeomorphic to the $3$-torus which is more or less immediate 	from the construction of $(K^4)_{16}$ as the 4-torus modulo the 	central involution. 	The other slicings are connected sums of 	$\mathbb{R}P^3$ and $S^2 \times S^1 $. 	They are listed in Table \ref{tab:toptypeskummer}, 	the multiplicity vectors are shown in Table \ref{tab:kummer}.
	
	\begin{table}[H]
		\begin{center}
			\setlength{\extrarowheight}{0.1cm}
			\begin{tabular}{|c|c|c|} \hline
				level of $f_{\{1 , \ldots , 16 \}}$& type & slicing in between \\
				\hline
				$\frac{1}{30}$&$\mathbb{R}P^3 $&$\{1\}$ and $\{2,\ldots ,16\}$\\
				$\frac{1}{10}$&$\mathbb{R}P^3\#\mathbb{R}P^3$&$\{1,2\}$ 	and $\{3,\ldots ,16\}$\\
				$\frac{1}{6}$&$\mathbb{R}P^3\#\mathbb{R}P^3\# \mathbb{R}P^3$&$\{1,\ldots ,3\}$ and $\{4,\ldots ,16\}$\\
				$\frac{7}{30}$&$(S^2\times S^1)\# 2(\mathbb{R}P^3)$&	$\{1,\ldots ,4\}$ and $\{5,\ldots ,16\}$\\
				$\frac{1}{2}$&$\mathbb{T}^3$&$\{1,\ldots ,8\}$ and $\{9,\ldots ,16\}$\\[0.1cm] 
				\hline	 
			\end{tabular}
			\\[0.2cm]
			\caption{Slicings of $(K^4)_{16}$ by the perfect and symmetric 	rsl-function $f_{\{1 , \ldots , 16 \}}$\label{tab:toptypeskummer}}
		\end{center}
	\end{table}
	
	An example of an rsl-function which is not a perfect function
	is the function 
	\[ \textstyle f_{\{1,4,6,2,3,5,7,\ldots,16\}} : (K^4)_{16} \to \left[ 0 , 1
	\right]. \]
	This admits an empty triangle 	on one side leading to a critical point of index 1. 	In fact $f_{\{1,4,6,2,3,5,7,\ldots,16\}}$ has precisely 20 critical points, counted 	with multiplicity, see Table \ref{tab:kummer}.
\end{proof}

\begin{table}[H]
	{\small 
	\begin{center}
		\setlength{\extrarowheight}{0.1cm}
		\begin{tabular}{|c||cc||cc|} \hline
			& \multicolumn{2}{|l||}{$f_{\{1,\ldots,16\}}$}& \multicolumn{2}{|l|}{$f_{\{1,4,6,2,3,5,7,\ldots,16\}}$} \\ 
			\hline
			level&$v$&$\mathbf{m}(v,\mathbb{F}_2)$&$v$&$\mathbf{m}(v,\mathbb{F}_2)$\\	
			\hline
			$0$&$1$&$(1,0,0,0,0)$&$1$&$ (1,0,0,0,0) $\\
			$\frac{1}{15}$&$2$&$(0,0,0,0,0)$&$4$&$ (0,0,0,0,0) $\\
			$\frac{2}{15}$&$3$&$(0,0,0,0,0)$&$6$&$ (0,1,0,0,0) $\\
			$\frac{3}{15}$&$4$&$(0,0,1,0,0)$&$2$&$ (0,0,1,0,0) $\\
			$\frac{4}{15}$&$5$&$(0,0,0,0,0)$&$3$&$ (0,0,1,0,0) $\\
			$\frac{5}{15}$&$6$&$(0,0,1,0,0)$&$5$&$ (0,0,1,0,0) $\\
			$\frac{6}{15}$&$7$&$(0,0,1,0,0)$&$7$&$ (0,0,1,0,0) $\\
			$\frac{7}{15}$&$8$&$(0,0,1,1,0)$&$8$&$ (0,0,1,1,0) $\\
			$\frac{8}{15}$&$9$&$(0,0,0,0,0)$&$9$&$ (0,0,0,0,0) $\\
			$\frac{9}{15}$&$10$&$ (0,0,1,0,0)$&$10$&$ (0,0,1,0,0) $\\
			$\frac{10}{15}$&$11$&$(0,0,1,0,0)$&$11$&$ (0,0,1,0,0) $\\
			$\frac{11}{15}$&$12$&$(0,0,1,1,0)$&$12$&$ (0,0,1,1,0) $\\
			$\frac{12}{15}$&$13$&$(0,0,1,0,0)$&$13$&$ (0,0,1,0,0) $\\
			$\frac{13}{15}$&$14$&$(0,0,1,1,0)$&$14$&$ (0,0,1,1,0) $\\
			$\frac{14}{15}$&$15$&$(0,0,1,1,0)$&$15$&$ (0,0,1,1,0) $\\
			$1$&$16$&$(0,0,1,1,1)$&$16$&$ (0,0,1,1,1) $\\ 
			\hline
		\end{tabular}
		\\[0.2cm]
		\caption{Multiplicity vectors of two rsl-functions $f_{\{1,\ldots,16\}}$ and $f_{\{1,4,6,2,3,5,7,\ldots,16\}}$ on $(K^4)_{16}$ \label{tab:kummer}}
	\end{center}}
\end{table}

\section{Further Results}
\label{other}

In each of the 16 steps of the dilatation process for the Kummer variety we have the choice between two orientations. Consequently for the resulting non-singular manifold at the end there are a number of different topological types which are possible. One can describe these by the intersection form.

\begin{prop}
	One can construct some combinatorial $4$-manifolds realizing any of the intersection forms of rank $22$ and signature $2n$, $n \in \{ 0 , \ldots , 8 \} $ from the triangulated $4$-dimensional Kummer variety $K^4$ by $16$ simplicial blowups, except for $19(\mathbb{C}P^2) \# 3(-\mathbb{C}P^2)$ and, possibly, $11(S^2 \times S^2)$. 
\end{prop}

\begin{proof}
	The case $n=8$ was already treated in Section~\ref{sec:K4toK3}. In this case the orientation was uniquely determined in every step by the one in the first step. Therefore the construction is essentially unique (up to the orientation in the first blowup) and leads to the $K3$ surface. In particular the manifold $19(\mathbb{C}P^2) \# 3(-\mathbb{C}P^2)$ cannot be obtained in this way.

	\medskip
	The signature of an even intersection form of a simply connected PL $4$-manifold is divisible by $16$ by Rohlin's Thm., cf. \cite{Freedman1976}. It follows that for $n \in \{ 1,\ldots ,7 \} $ we have an odd intersection form. In these cases the manifold is homeomorphic to
	\begin{equation*}
	 k(\mathbb{C}P^2) \# l(- \mathbb{C}P^2)
	\end{equation*}
	where $ k - l = \pm 2n $, $n \in \{ 1 , \ldots , 7 \} $ and $k+l=22$. In the case $n=0$ the construction is not unique: The pattern of the orientations of all $16$ blowups is not determined since there are $8$ positive and $8$ negative blowups distributed arbitrarily in $K^4$. An odd intersection form was obtained by one particular sequence. This leads to the manifold $11(\mathbb{C}P^2) \# 11(-\mathbb{C}P^2)$.
\end{proof}

\medskip 
The question whether or not the manifold $11(S^2 \times S^2)$ can also be obtained by this construction remains open at this point. It must also be left open whether or not any of the other manifolds with a 22-dimensional second homology admits a triangulation with only 16 vertices. By \cite[Thm.4.9]{Kuehnel95TightPolySubm} such a 16-vertex triangulation would have to be 3-neighborly and would, by the Dehn-Sommerville equations, have the same $f$-vector as $(K3)_{16}$ and, thus, would give a solution to Problem \ref{prob2} in the introduction. Further experiments in this direction could possibly produce such an example. This is still work in progress.

In the case of $10$ vertices and $\chi = 4$ the combinatorial data corresponds to three topological types of simply connected $4$-manifolds, namely $S^2 \times S^2$, $\mathbb{C}P^2 \# \mathbb{C}P^2$ and $\mathbb{C}P^2 \# (-\mathbb{C}P^2)$. These would be candidates for a solution to Problem \ref{prob1}. However, it was shown in \cite{Kuehnel83Uniq3Nb4MnfFewVert} that in fact none of the topological manifolds above has a combinatorial triangulation with only $10$ vertices.

\medskip
More details of the combinatorial processes described above are available from \texttt{http://www.igt.uni-stuttgart.de/LstDiffgeo/Spreer/k3}. Moreover, most of the algorithms used to compute simplicial blowups and multiplicity vectors of rsl-functions are planned to be available soon within the GAP-package \texttt{simpcomp} \cite{simpcomp}, maintained by Effenberger and the first author. 

\bigskip
Acknowledgement: This work was partially 
supported by the Deutsche Forschungsgemeinschaft (DFG)
under the grant Ku 1203/5-2.

\addcontentsline{toc}{chapter}{Bibliography}
\bibliographystyle{amsplain}
\setlength{\baselineskip}{0pt}
\providecommand{\bysame}{\leavevmode\hbox to3em{\hrulefill}\thinspace}
\providecommand{\MR}{\relax\ifhmode\unskip\space\fi MR }
\providecommand{\MRhref}[2]{%
 \href{http://www.ams.org/mathscinet-getitem?mr=#1}{#2}
}
\providecommand{\href}[2]{#2}

Institut f\"ur Geometrie und Topologie

Universit\"at Stuttgart

70550 Stuttgart

Germany
\end{document}